\documentclass[11pt, letterpaper]{article}
\usepackage{fullpage}
\usepackage{amsthm}
\usepackage{amsmath,amssymb,amsfonts,nicefrac}
\usepackage{xspace}
\usepackage{color}
\usepackage{url}
\usepackage{hyperref}
\usepackage{bm}
\usepackage{bbm}
\usepackage{times}

\usepackage{enumitem}

\newtheorem{thm}{Theorem}[section]

\newtheorem{lemma}[thm]{Lemma}

\newtheorem{claim}[thm]{Claim}

\newtheorem{definition}[thm]{Definition}

\newtheorem{fact}[thm]{Fact}

\newcommand\E{\mathop{\mathbb{E}}}
\newcommand\card[1]{\left| {#1} \right|}
\newcommand\sett[2]{\left\{ \left. #1 \;\right\vert #2 \right\}}

\newcommand\Prob[2]{{\Pr_{#1}\left[ {#2} \right]}}

\newcommand\cProb[3]{{\Pr_{#1}\left[ \left. #3 \;\right\vert #2 \right]}}

\newcommand\Expect[2]{{\mathop{\mathbb{E}}_{#1}\left[ {#2} \right]}}

\newcommand\norm[1]{\| #1 \|}

\newcommand\inner[2]{\left\langle{#1},{#2}\right\rangle}
\newcommand\eps{\varepsilon}

\newcommand*\xor{\mathbin{\oplus}}

\renewcommand\geq{\geqslant}
\renewcommand\leq{\leqslant}

\newcommand{\rom}[1]{\uppercase\expandafter{\romannumeral #1\relax}}
\renewcommand\epsilon{\eps}

\title{A Dense Model Theorem for the Boolean Slice}
\author{
Gil Kalai
\footnote{Einstein Institute of Mathematics, Hebrew University of Jerusalem and Efi Arazi School of Computer Science, Reichman University. Supported by ERC grant 834735 and by an ISF grant 2669/21.}
\and
Noam Lifshitz
\footnote{Einstein Institute of Mathematics, Hebrew University of Jerusalem. Supported by ISF grant  1980/22.}
\and
Dor Minzer
\footnote{Department of Mathematics, Massachusetts Institute of Technology. Supported by NSF CCF award 2227876 and NSF CAREER award 2239160.}
\and
Tamar Ziegler
\footnote{Einstein Institute of Mathematics, Hebrew University of Jerusalem. Supported by ISF grant 2112/20.}
}
\date{\vspace{-5ex}}
\begin{document}
\maketitle
\begin{abstract}
The (low soundness) linearity testing problem for the middle slice of the Boolean cube is as follows. Let $\eps>0$ and $f$ be a function on the middle slice on the Boolean cube, such that when choosing a uniformly random quadruple $(x,y,z ,x\oplus y\oplus z)$ of vectors of $2n$ bits with exactly $n$ ones, the probability that $f(x\oplus y \oplus z) = f(x) \oplus f(y) \oplus f(z)$ is at least $1/2+\eps$. The linearity testing problem, posed by~\cite{DBLP:journals/siamcomp/DavidDGKS17}, asks whether there must be an actual linear function that agrees with $f$ on $1/2+\eps'$ fraction of the inputs,
where $\eps' = \eps'(\eps)>0$.

We solve this problem, showing that $f$ must indeed be correlated with a linear function. To do so, we prove a dense model theorem for the
middle slice of the Boolean hypercube
for Gowers uniformity norms. Specifically,
we show that for every $k\in\mathbb{N}$,
the normalized indicator function of the middle slice of the Boolean hypercube
$\{0,1\}^{2n}$ is close in Gowers norm to
the normalized indicator function of
the union of all slices with weight
$t = n\pmod{2^{k-1}}$. Using our techniques we also give a more general `low degree test' and a biased
rank theorem for the slice.
\end{abstract}

\section{Introduction}
The middle slice of the Boolean hypercube $\{0,1\}^{2n}$ is given by $\mathcal{U}_{2n} = \sett{x\in \{0,1\}^{2n}}{\card{x} = n}$. A fundamental problem in a subarea of theoretical computer science known as property testing concerns determining whether properties of functions can be determined efficiently by examining only few of their inputs. One striking classical result of this kind is the BLR theorem~\cite{blum1990self} which shows that the question whether a function $f\colon \mathbb{F}_2^n\to \mathbb{F}_2$ is linear or far from any linear function can be determined by evaluation $f$ on only three random inputs of the form $(x,y,x\oplus y)$.
David, Dinur, Goldenberg, Kindler, and Shinkar~\cite{DBLP:journals/siamcomp/DavidDGKS17} solved the corresponding problem for the slice by showing that if $(x,y,x\oplus y)$ are a uniformly random triple among such triples with $x,y,x\oplus y\in \mathcal{U}_{2n}$ and $f\colon \mathcal{U}_{2n}\to \mathbb{F}_2$ satisfies $\Pr[f(x\oplus y) = f(x) \oplus f(y)]\ge 1  - \epsilon$, then $f$ agrees with a linear function on $1-O(\epsilon)$ fraction of the inputs.

They then posed the corresponding problem in the so called `low soundness' regime, also known as the 1\% regime, where one wishes to show that if a function passes the test with probability significantly larger than $1/2$, then $f$ is correlated with a linear function. The main result
of this paper is a result
along these lines of a
closely related test.

\begin{thm}\label{thm:linearity}
    For all $\eps>0$ there exists
    $n_0 = n_0(\eps)\in\mathbb{N}$
    such that the following holds for $n\geq n_0$.
    If a function
    $f\colon \mathcal{U}_{2n}\to\{0,1\}$
    satisfies that
    \[
        \cProb{x,y,z}{x,y,z,x\oplus y\oplus z\in\mathcal{U}_{2n}}{f(x)\oplus f(y)\oplus f(z) = f(x\oplus y\oplus z)}
        \geq \frac{1}{2}+\eps,
    \]
    then there exists $S\subseteq [2n]$
    and $b\in\{0,1\}$  such that
    \[
    \cProb{x}{x\in\mathcal{U}_{2n}}{f(x) = b+\bigoplus_{i\in S}x_i}
    \geq \frac{1}{2} + \frac{\sqrt{\eps}}{400}.
    \]
\end{thm}

\subsection{Method}
Intuitively, one may think of
the Boolean slice as
very similar
to the Boolean hypercube, as one
only imposes a constraint on the Hamming weight of the vectors. Indeed, there is some truth to it; as far as low-degree functions are concerned (more specifically, degree $o(\sqrt{n})$ functions), the two domains are almost
interchangeable; this is formalized in~\cite{FilmusMossel,FKMW} as an
invariance principle (see also~\cite{BKLM}
for a simplification and extensions).
In general however, proving results
for the slice is more challenging, even when these results are
concerned with low-degree functions; see for example~\cite{Filmus1,Filmus2}.
For results regarding high degree
functions, the situation is even worse as there are high-degree functions that certainly
distinguish between the two domains, and therefore the argument
in the slice has to be significantly
different from the argument on
the hypercube.
The problem of linearity testing
over the slice is an example
problem which
concerns high degree functions (in the sense of Fourier analytic degree), and as mentioned
prior to this paper only partial results
in the $99\%$ regime were known~\cite{DBLP:journals/siamcomp/DavidDGKS17}.

Our main tool is a variation of the dense
model theorem, which was famously used by Green and Tao \cite{green2008primes} to show the primes contain an arbitrary long arithmetic progression. Roughly speaking, the dense model theorem allows transferring arithmetic properties of dense subsets of an Abelian group $G$ to analogue properties of dense subsets of arbitrary sets $S\subseteq G$ with the property that $\frac{|G|1_S}{|S|}$ has a small $U_k$-norm.

The dense model provides us with another
avenue of comparing the Boolean slice
$\mathcal{U}_{2n}$ with the Boolean
hypercube and deducing results about
the former from the latter.
Here and throughout, for an integer
$s\geq 1$, the Gowers uniformity norm $U_s$ of a function $f\colon\{0,1\}^n\to\mathbb{C}$ is defined as
\[
\norm{f}_{U_s}
=\left(\Expect{x,h_1,\ldots,h_s\in\{0,1\}^n}{\prod\limits_{T\subseteq [s]}C^{|T|}f\left(x+\sum\limits_{i\in T}h_i\right)}\right)^{\frac{1}{2^s}},
\]
where $C$ denoted complex conjugation.
For $k\in\mathbb{N}$ thought of as a constant, suppose that $n = a\pmod{2^{k-1}}$ and define
\[
\mathcal{D}_{2n,k} = \sett{x\in\{0,1\}^{2n}}{\card{x} = a\pmod{2^{k-1}}}.
\]
Note that $\mathcal{D}_{2n,k}\supseteq \mathcal{U}_{2n}$, and
while the measure of $\mathcal{U}_{2n}$ is vanishing with $n$
(more specifically, it is $\Theta(1/\sqrt{n})$), the measure
of $\mathcal{D}_{2n,k}$ is constant (roughly standing at $\frac{1}{2^{k-1}}+o(1)$).
The main tool of our paper establishes that $\mathcal{D}_{2n,k}$ is
a dense model for $\mathcal{U}_{2n}$ for Gowers' uniformity
norms, and more precisely:
\begin{thm}\label{thm:main}
    For all $k\in \mathbb{N},\epsilon >0$ there exists $n_0>0$, such that if $n>n_0$ we have
    \[
    \left\|
    \frac{1_{\mathcal{U}_{2n}}}{\E[1_{\mathcal{U}_{2n}}]}
    -
    \frac{1_{\mathcal{D}_{2n,k}}}{\E[1_{\mathcal{D}_{2n.k}}]}
    \right\|_{U_{k}} \leq \epsilon.
    \]
\end{thm}
A variation of the Green Tao dense model theorem allows us to deduce Theorem \ref{thm:linearity} from Theorem \ref{thm:main}. Indeed, we
use the dense model theorem of~\cite{dodos2022uniformity}.
\footnote{We remark that for $k=2$ one could use the result of Conlon, Fox and Zhao~\cite{conlon2014green}, but the linear
forms condition becomes difficult to check
for $k\geq 3$.}

\subsection{Other applications}
\subsubsection{Low Degree Testing}
A second application of our method concerns the problem of testing polynomial of  higher-degree, also in the low-soundness regime.
For an integer $d\in\mathbb{N}$,
consider the $d$-Gowers' test over the slice:
\begin{enumerate}
    \item Sample
    $x,h_1,\ldots,h_{d}\in\{0,1\}^{2n}$
    conditioned on $x\oplus \bigoplus_{i\in T} h_i\in \mathcal{U}_{2n}$
    for all $T\subseteq [d]$.
    \item Check that
    $\sum\limits_{T\subseteq[d]}f\left(x\oplus\bigoplus_{i\in T}h_i\right) = 0$.
\end{enumerate}
We prove the following result, asserting
that if a Boolean function $f$ passes
the $d$-Gowers' test with probability
$1/2 + \eps$, then it must be correlated
with a non-classical polynomial of degree
at most $d-1$.\footnote{We remark that
a random function passes the test with
probability $1/2$, and therefore the
natural question is what can be said
about a function that passes the test
with probability noticeably bigger than
$1/2$.}
\begin{thm}
\label{thm:gowers_low_soundness}
    For all $\eps>0$, $d\in\mathbb{N}$
    there are $n_0\in\mathbb{N}$ and
    $\delta>0$ such that the following holds for $n\geq n_0$.
    Suppose that $f\colon \mathcal{U}_{2n}\to\{0,1\}$
    passes the $d$-Gowers' test with
    probability at least $\frac{1}{2}+\eps$.
    Then there exists a non-classical polynomial
    $p\colon \{0,1\}^{2n}\to[0,1)$
    of degree at most $d-1$ such that
    \[
    \card{
    \Expect{x\in\mathcal{U}_{2n}}
    {(-1)^{f(x)}e^{2\pi{\bf i}p(x)}}
    }
    \geq \delta.
    \]
\end{thm}
Theorem~\ref{thm:gowers_low_soundness} gives an
answer to a question of~\cite{DBLP:journals/siamcomp/DavidDGKS17} regarding degree $d$ testing over
the slice.

\subsubsection{Biased Rank}
A third application of our method is concerned with the biased rank problem. In this scenario, we show that
a low-degree polynomial $P$
on the slice which is biased must be of small
rank. More precisely,
\begin{thm}\label{thm:biased_rank}
    Let
    $P: \{0,1\}^{2n} \to [0,1)$ be a non-classical polynomial of degree $d$ such that
    $
    \card{\Expect{x\in \mathcal{U}_{2n}}{(-1)^{P(x)}}
    }\geq \delta$, and
    suppose that $n=a\pmod{2^d}$.
    Then there is $L=L(\delta, d)$ such that for some $j$ and some $\Gamma\colon [0,1)^{L}\to [0,1)$
    \[
    P(x) -\frac{j(\card{x}-a)}{2^d}= \Gamma(Q_1(x), \ldots Q_L(x))
    \pmod{1}
    \]
    for all $x\in\{0,1\}^{2n}$, where $Q_i$ are
    non-classical polynomials of degree strictly smaller than $d$.
    In particular,
    for $x\in U_{2n}$ we have that
    \[
    P(x) = \Gamma(Q_1(x), \ldots Q_L(x))
    \pmod{1}.
    \]
\end{thm}

\subsection{Other related works}
We remark that the Gowers' uniformity norms
have been used in theoretical computer science
in numerous other contexts, such as PCP~\cite{samorodnitsky2006gowers}, communication complexity~\cite{viola2007norms},
property testing~\cite{alon2003testing} and
more; we refer the reader to~\cite{hatami2019higher} for a survey.
There are also a few connections between
dense model theorems and theoretical computer
science~\cite{reingold2008dense,mironov2009computational,trevisan2011dense}.

\section{Preliminaries}
In this section we present a few necessary
facts from analysis of Boolean functions.
We refer the reader to~\cite{o2014analysis}
for a more systematic presentation.
We consider the Boolean hypercube
$\{0,1\}^n$ equipped with the uniform
measure, and define an inner product
for functions over $\{0,1\}^n$
as
\[
\inner{f}{g} = \Expect{x}{f(x)\overline{g(x)}}
\]
for all $f,g\colon \{0,1\}^n\to\mathbb{C}$. For each $S\subseteq [n]$ we may define
the Fourier character $\chi_S\colon \{0,1\}^n\to\{-1,1\}$ by
$\chi_S(x) = \prod\limits_{i\in S}(-1)^{x_i}$. It is a standard fact
that $\{\chi_S\}_{S\subseteq [n]}$
is an orthonormal basis for
$L_2(\{0,1\}^n)$, and
thus any $f\colon\{0,1\}^n\to\mathbb{R}$
can be written as
$f(x) = \sum\limits_{S\subseteq [n]}\widehat{f}(S)\chi_S(x)$
where $\widehat{f}(S) = \inner{f}{\chi_S}$.

\begin{definition}
For an integer $1\leq d\leq n$ and
a function $f\colon\{0,1\}^n\to\mathbb{R}$,
we define the level $d$ weight of $f$
to be
\[
W_{\leq d}[f] = \sum\limits_{\card{S}\leq d}\card{\widehat{f}(S)}^2.
\]
\end{definition}

Our proof uses the level $d$-inequality,
which asserts a function $f\colon \{0,1\}^n\to\{-1,0,1\}$ that has small $\ell_1$-norm must have small level $d$
weight.
\begin{lemma}\label{lem:lvl_d_inequality}
Suppose that $f\colon \{0,1\}^n \to \{-1,0,1\}$
has $\alpha = \E[|f|]$. Then for all $d\leq n$,
\[
W_{\leq d}[f]\leq \alpha^2 \log^{O(d)}(1/\alpha).
\]
\end{lemma}
Lastly, we will use the notion of discrete
derivatives defined as follows:
\begin{definition}\label{def:derv_mult}
    For a function $f\colon \{0,1\}^n\to\mathbb{C}$
    and a direction $h\in\{0,1\}^n$,
    the discrete derivative of $f$
    in direction $h$ is the function
    $\partial_h f\colon
    \{0,1\}^n\to\mathbb{C}$ defined by
    \[
    \partial_h f(x)
    =
    \overline{f(x\oplus h)}f(x).
    \]
    For directions $h_1,\ldots,h_d\in\{0,1\}^n$, we define
    \[
    \partial_{h_1,\ldots,h_d} f(x)
    =(\partial_{h_1}\partial_{h_2}
    \cdots \partial_{h_d}) f(x).
    \]
\end{definition}

\paragraph{Non-classical polynomials:}
Let $\mathbb{T} = \mathbb{R}/\mathbb{Z}$
be the torus.
A
function $p\colon\{0,1\}^n\to\mathbb{T}$
is called a degree $d$ non-classical polynomial if $\partial_{h_1,\ldots,h_{d+1}} p=0$ for all directions $h_1,\ldots,h_{d+1}\in\{0,1\}^n$; here
$\partial_h p(x) = p(x+h) - p(x)$
is the standard notion of discrete derivative. We will not use this
notion of derivative and hence there
will be no confusion regarding which
notion of derivative is used.

We note that if $p$
is a degree $d$ non-classical polynomial,
then $\partial_{h_1,\ldots,h_{d+1}}e^{2\pi{\bf i} p}\equiv 1$ for all all directions $h_1,\ldots,h_{d+1}\in\{0,1\}^n$;
the derivative now
is as in Definition~\ref{def:derv_mult}.
\section{Proof of Theorem~\ref{thm:main}}
\subsection{Auxiliary Tools}
We begin by presenting a few tools that we need in the proof of
our main result. We begin with the following fact, asserting
that if we know the Hamming weight of $x,z$ and $x \oplus z$
modulo $2^j$, then we know the Hamming weight of $x\land z$
modulo $2^{j-1}$.
\begin{fact}\label{fact:trivial_inter}
Suppose that $\card{x\oplus z} = b\pmod{2^j}$, $\card{x} = c\pmod{2^j}$ and $\card{z} = d\pmod{2^j}$. Then
$c+d-b$ is divisible by $2$ and
$\card{x\land z} = \frac{c+d-b}{2}\pmod{2^{j-1}}$.
\end{fact}
\begin{proof}
Note that
$\card{x\oplus z}
= \card{x\lor z} - \card{x\land z}
= \card{x}+\card{z} - 2\card{x\land z}
$,
and the result follows form re-arranging and dividing by $2$.
\end{proof}

It will be convenient for us
to think of a vector $x\in\{0,1\}^{2n}$ also as a subset of $[2n]$, namely as ${\sf supp}(x)$.
For vectors $x_1,\ldots,x_t\in \{0,1\}^{2n}$, the algebra generated by them $\mathcal{B} = \mathcal{B}[{\sf supp}(x_1),\ldots,{\sf supp}(x_t)]$
consists of vectors that correspond to sets that can
be formed by ${\sf supp}(x_1),\ldots,{\sf supp}(x_t)$
by taking unions, intersections and complements (in other words, it is the $\sigma$-algebra generated by these sets). Suppose that $x_1,\ldots,x_t\in \{0,1\}^{2n}$ are vectors
so that in $\mathcal{B} = \mathcal{B}[{\sf supp}(x_1),\ldots,{\sf supp}(x_t)]$, each
atom has at least $n/2^{t+1}$ elements. It is a standard computation that sampling $x\in\{0,1\}^{2n}$ uniformly we have
that $x\in\mathcal{U}_{2n}$ with probability $\Theta(1/\sqrt{n})$.
We would like to say that the events that $x\oplus z\in\mathcal{U}_{2n}$ for all $z\in{\sf Span}(x_1,\ldots,x_t)$
are almost indepednent, and in particular that
\[
\Prob{x}{x\oplus z\in\mathcal{U}_{2n}\text{ for all }z\in{\sf Span}(x_1,\ldots,x_t)}\leq O_t(\sqrt{n}^{-2^t}).
\]
The following lemma shows that this is indeed the case.
\begin{lemma}\label{lem:all_in_slice}
In the setting above, for all $n_z\in\mathbb{N}$
\[
\Prob{x}{\card{x\oplus z} = n_z\text{ for all }z\in{\sf Span}(x_1,\ldots,x_t)}\leq O_t(n^{-2^{t-1}}).
\]
\end{lemma}
\begin{proof}
Let $B\in\mathcal{B}$ be an atom, and without loss of generality
$B = {\sf supp}(x_1)\cap\ldots \cap{\sf supp}(x_t)$. We show that
if this event holds, then $\card{x\land x_1\land\ldots\land x_t}$
has to be equal to some specific value. Indeed, note that as
in Fact~\ref{fact:trivial_inter} for all $z\in{\sf Span}(x_1,\ldots,x_t)$
\[
\card{x\land z} =
\frac{1}{2}\left(\card{x} + \card{z} - \card{x\xor z}\right),
\]
hence the Hamming weight of
$x\land z$ has to be a specific value if the event
in question holds. Also, we know that for $z\in{\sf Span}(x_2,\ldots,x_t)$ we have
\[
\card{x\land (x_1 \xor z)}
=\card{x\land (x_1\cup z)}
-\card{x\land x_1\land z}
=\card{x\land x_1} + \card{x\land z}
-2\card{x\land x_1\land z},
\]
and hence the Hamming weight of $x\land x_1\land z$ is determined.
Noting that
\[
\card{(x\land x_1)\xor z}
=
\card{x\land x_1} + \card{z} - 2\card{x\land x_1\land z}
\]
we get that the Hamming weight of $(x\land x_1)\xor z$ is determined
for all $z\in{\sf Span}(x_2,\ldots,x_t)$. Iterating the argument
with $x' = x\land x_1$ gives that the Hamming weight of $(x'\land x_2)\xor z$ is determined for all $z\in{\sf Span}(x_3,\ldots,x_t)$, and repeating this argument shows that
the Hamming weight of $\card{x\land x_1\land \ldots\land x_t}$
is determined.

For each $b\in \{-1,1\}^t$, let $x_b\in \{0,1\}^{2n}$
be the vector $\bigwedge_{i=1}^{t} v_i$ where $v_i = x_i$
if $b_i = 1$ and else $v_i = 1-x_i$. From the above argument,
it follows that for each $b$ there is a number $n_b$ such that
if the event in question holds, then $\card{x\land x_b} = n_b$.
Thus, the probability in question is at most
\[
\Prob{x}{\card{x\land x_b} = n_b~\forall b\in \{-1,1\}^t}
=\prod\limits_{b\in \{-1,1\}^t}\Prob{x}{\card{x\land x_b} = n_b},
\]
where the last transition holds because the supports of the vectors $x_b$ are disjoint and so the random variables $x\land x_b$ are independent. As the support of $x_b$ is in the algebra
generated by $x_1,\ldots,x_t$, we have that $\card{x_b} \geq n/2^{t+1}$ and hence
\[
\Prob{x}{\card{x\land x_b} = n_b}
\leq O\left(\sqrt{\frac{2^t}{n}}\right),
\]
and plugging this above finishes the proof.
\end{proof}
We will also need to consider sub-events of the event in
Lemma~\ref{lem:all_in_slice} and argue that they are also
almost independent. More precisely:
\begin{lemma}\label{lem:slightly_notrivial_bound}
Let $Z\subseteq {\sf Span}(x_1,\ldots,x_t)$. Then
\[
\Prob{x}{x\oplus z\in\mathcal{U}_{2n}~\forall z\in Z}\leq
n^{-\card{Z}/2} (\log n)^{O_t(1)}.
\]
\end{lemma}
\begin{proof}
  Consider $z'\in \overline{Z}:= {\sf Span}(x_1,\ldots,x_t)\setminus Z$,
  and note that by Chernoff's bound the probability that $\card{x\oplus z'}$ is outside
  the range between $n - \sqrt{n}\log n$ and $n + \sqrt{n}\log n$
  is at most $2^{-\Omega(\log^2 n)}$. Thus, the probability
  in question is at most
  \begin{equation}\label{eq3}
    \Prob{x}{x\oplus z\in\mathcal{U}_{2n}~\forall z\in Z,
    \card{x\oplus z'} \in [n-\sqrt{n}\log n,n+\sqrt{n}\log n]
    ~\forall z'\in \overline{Z}} + 2^{t-\Omega(\log^2 n)}.
  \end{equation}
  Denote $q = \Prob{x}{x\oplus z\in\mathcal{U}_{2n}~\forall z\in Z,
  \card{x\oplus z'} \in [n-\sqrt{n}\log n,n+\sqrt{n}\log n]~\forall z'\in\overline{Z}}$.
  Trivially,
  \[
    q =
    \sum\limits_{
    (n_{z'})_{z'\in\overline{Z}}\in [n - \sqrt{n}\log n, n+\sqrt{n}\log n]^{\card{\overline{Z}}}}
    \Prob{x}{x\oplus z\in\mathcal{U}_{2n}~\forall z\in Z, \card{x\oplus z'} = n_{z'}~\forall z'\in \overline{Z}},
  \]
  and as by Lemma~\ref{lem:all_in_slice} each summand is at most $O_{t}(n^{-2^{t-1}})$  we get that
  \[
  q\leq (2\sqrt{n}\log n + 1)^{\card{\overline{Z}}}O_{t}(n^{-2^{t-1}})
  \leq n^{-\card{Z}/2}(\log n)^{O_t(1)}.
  \]
  Plugging this into~\eqref{eq3} finishes the proof.
\end{proof}

\subsection{The Main Argument}
 The proof of Theorem~\ref{thm:main} is by induction. To
 be more precise, fix $n$ and $k$ as in the theorem. We prove
 by induction on $k'$ the following result:
 \begin{thm}\label{thm:main'}
     For all $1\leq k'\leq k$ we have that
     \[
     \left\|
     \frac{1_{\mathcal{U}_{2n}}}{\E[1_{\mathcal{U}_{2n}}]}
     -
     \frac{1_{\mathcal{D}_{2n,k}}}{\E[1_{\mathcal{D}_{2n,k}}]}
     \right\|_{U_{k'}}^{2^{k'}} \leq \frac{(\log n)^{C_{k',k}}}{n},
     \]
     where $C_{k',k}$ is a constant depending only on $k',k$.
 \end{thm}
 The base case of Theorem~\ref{thm:main'}, namely the case that
 $k'=1$, is clear as the $U_1$ norm is the absolute value of the
 expected value of the function, which in our case is $0$.
 We now move on to the inductive hypothesis and fix $k'\geq 2$
 and assume the statement holds for $k'-1$. We show that the
 statement holds for $k'$ with $C_{k',k}\leq C_{k'-1,k} + O_{k}(1)$. Throughout this section,
 we will denote $\mathcal{D}_{2n} = \mathcal{D}_{2n,k}$ for ease of notations.

For simplicity of presentation, denote
$f = \frac{1_{\mathcal{U}_{2n}}}{\E[1_{\mathcal{U}_{2n}}]}$ and
$g = \frac{1_{\mathcal{D}_{2n}}}{\E[1_{\mathcal{D}_{2n}}]}$.
By definition of Gowers' uniformity norms, we have that
\[
\norm{f-g}_{U_{k'}}^{2^{k'}}
=\Expect{x_1,\ldots,x_{k'-2}}
{\norm{\partial_{x_1,\ldots,x_{k'-2}}(f-g)}_{U_2}^4}.
\]
Consider the algebra $\mathcal{B}$ generated by ${\sf supp}(x_i)$ for
$i=1,\ldots,k'-2$, and let $E$ be the event that each
atom there has at least $n/2^{k}$ elements. By Chernoff's
bound we have that $\Prob{}{\bar{E}}\leq 2^{-\Omega(n)}$,
and as $\norm{f}_{\infty},\norm{g}_{\infty}\leq n$ we get
that
\begin{align}\label{eq1}
\norm{f-g}_{U_{k'}}^{2^{k'}}
&\leq
2^{-\Omega(n)}n^{\Theta(2^k)}
+
\Expect{x_1,\ldots,x_{k'-2}}{\norm{\partial_{x_1,\ldots,x_{k-2}}(f-g)}_{U_2}^41_E}\notag\\
&\leq O_k\left(\frac{1}{n}\right)
+
\Expect{x_1,\ldots,x_{k'-2}}{1_E
\cdot \sum\limits_{S}\card{\widehat{\partial_{x_1,\ldots,x_{k'-2}}(f-g)}(S)}^4}.
\end{align}
We focus on the second sum now. Consider the function
$h(x) = \partial_{x_1,\ldots,x_{k'-2}}(f-g)(x)$, and
define the subgroup $G\subseteq S_{2n}$ as
\[
G = \sett{\pi\in S_{2n}}{\pi(B) = B~\forall B\in\mathcal{B}},
\]
and note that $h$ is symmetric under $G$. For a subset $I\subseteq [2n]$, define
\[
{\sf orb}_{I}
=\sett{\pi(I)}{\pi\in G},
\]
so that we get that $\widehat{h}(I) = \widehat{h}(I')$
if $I'\in {\sf orb}_I$. In our upper bound of~\eqref{eq1}
we partition the characters $S$ into $3$ types, and upper
bound the contribution of each one of them separately.
This partition will depend on the size of the orbit of
the character $S$ under $G$, and towards this end we have
the following claim.
\begin{claim}\label{claim:orbit_estimates}
Let $S\subseteq [2n]$, and let $d = \min_{B\in\mathcal{B}}\card{S\Delta B}$.
\begin{enumerate}
    \item If $d=0$, then $\card{{\sf orb}_S}=1$.
    \item If $1\leq d\leq 2^{100\cdot k}$, then
    $\card{{\sf orb}_S}\geq \frac{n}{2^{101 k}}$.
    \item If $d > 2^{100\cdot k}$, then
    $\card{{\sf orb}_S}\geq \Omega_k(n^{100\cdot 2^k})$.
\end{enumerate}
\end{claim}
\begin{proof}
For the first item, if $d=0$ then $S\in\mathcal{B}$,
and it is clear that ${\sf orb}_S = \{S\}$.

For the
second item, write $S = B\Delta I$ where $B\in\mathcal{B}$
and $\card{I} = d$. Thus, $\pi(S) = B\Delta \pi(I)$ for
$\pi\in G$. Take some $i\in I$ and consider the atom
of $\mathcal{B}$ in which $i$ lies, say $B'$.
For any $j\in B'$ we may find $\pi_{i,j}\in G$
such that $\pi_{i,j}(i) = j$. As $\card{B'}\geq n/2^k$,
we conclude that there are at least $\frac{n}{d 2^k}$ distinct
sets among $B\Delta \pi_{i,j}(I)$.

For the third item, write $S = B\Delta I$ again for $B\in\mathcal{B}$ and $I$ of size $d$. Write $I = \{i_1,\ldots,i_d\}$, and let $B_1,\ldots,B_d$ be the atoms
of $\mathcal{B}$ that $i_1,\ldots,i_d$ lie in, respectively.
For all distinct $j_1\in B_1,\ldots,j_d\in B_d$ we may find $\pi_{i_1,\ldots,i_d,j_1,\ldots,j_d}\in G$
such that $\pi_{i_1,\ldots,i_d,j_1,\ldots,j_d}(i_{\ell}) = j_{\ell}$ for all $\ell=1,\ldots,d$. It follows that the orbit
of $S$ has size at least $\frac{n}{2^k}\cdot \left(\frac{n}{2^k} - 1\right)\cdots \left(\frac{n}{2^k}-100\cdot 2^{k}\right)\geq \Omega_k(n^{100\cdot 2^k})$.
\end{proof}

Define
\[
T_1 = \mathcal{B},
\qquad
T_2 = \sett{S\subseteq [2n]}{1\leq \min_{B\in\mathcal{B}}\card{S\Delta B}\leq 2^{100\cdot k}},
\qquad
T_3 = \sett{S\subseteq [2n]}{\min_{B\in\mathcal{B}}\card{S\Delta B}> 2^{100\cdot k}},
\]
so that by~\eqref{eq1} we have that
\begin{align}\label{eq2}
\norm{f-g}_{U_{k'}}^{2^{k'}}
&\leq
\Expect{x_1,\ldots,x_{k'-2}}{1_E
\cdot \sum\limits_{S\in T_1}\card{\widehat{h}(S)}^4}
+
\Expect{x_1,\ldots,x_{k'-2}}{1_E
\cdot \sum\limits_{S\in T_2}\card{\widehat{h}(S)}^4}\notag\\
&+
\Expect{x_1,\ldots,x_{k'-2}}{1_E
\cdot \sum\limits_{S\in T_3}\card{\widehat{h}(S)}^4}
+O_k\left(\frac{1}{n}\right).
\end{align}
We now upper bound each term on the right hand side of~\eqref{eq2}. We start with the contribution from $T_2$.

\begin{claim}\label{claim:T2}
   $\Expect{x_1,\ldots,x_{k-2}}{1_E
\cdot \sum\limits_{S\in T_2}\card{\widehat{h}(S)}^4} \leq \frac{(\log n)^{O_k(1)}}{n}$.
\end{claim}
\begin{proof}
Clearly we have that
\[
\Expect{x_1,\ldots,x_{k'-2}}{1_E
\cdot \sum\limits_{S\in T_2}\card{\widehat{h}(S)}^4}
= \Expect{x_1,\ldots,x_{k'-2}}{1_E
\cdot \sum\limits_{B\in \mathcal{B}}
\sum\limits_{1\leq \card{S\Delta B}\leq 2^{100\cdot k}}\card{\widehat{h}(S)}^4}.
\]
Denote $W_B = \sum\limits_{1\leq \card{S\Delta B}\leq 2^{100\cdot k}}\card{\widehat{h}(S)}^2$. As by Claim~\ref{claim:orbit_estimates} the orbit of each $S$
has size at least $n/2^{101 k}$, we get that
\begin{equation}\label{eq4}
\Expect{x_1,\ldots,x_{k'-2}}{1_E
\cdot \sum\limits_{S\in T_2}\card{\widehat{h}(S)}^4}
\leq
\frac{2^{101k}}{n}\Expect{x_1,\ldots,x_{k'-2}}{1_E
\cdot \sum\limits_{B\in \mathcal{B}}W_B^2}.
\end{equation}
Fix $x_1,\ldots,x_{k'-2}$ satisfying $E$ and $B\in\mathcal{B}$,
and inspect $W_B$.
Opening up the definition of $h$, by the triangle inequality  $\card{\widehat{h}(S)}$ is upper bounded by a sum of $2^{2^{k'-2}}$ terms of
the form $\card{\Expect{x}{F_1(x)\cdots F_{2^{k'-2}}(x)\chi_S(x)}}$
where writing ${\sf Span}(x_1,\ldots,x_{k'-2}) =
\{z_1,\ldots,z_{2^{k'-2}}\}$, each $F_i$ is either $f(x+z_i)$
or $g(x+z_i)$. Thus, by Cauchy-Schwarz
\[
W_B\leq 2^{2^{k'-2}}\sum\limits_{F_1,\ldots,F_{2^{k'-2}}}
W_{\leq 2^{100\cdot k}}[F_1\cdots F_{2^{k'-2}}\chi_B].
\]
Fix a choice of $F_1,\ldots,F_{2^{k-2}}$. Namely, fix a subset
$Z\subseteq {\sf Span}(x_1,\ldots,x_{k-2})$, and say that
\[
F_1\cdots F_{2^{k-2}}(x) = \prod\limits_{z\in Z}f(x\oplus z)\prod\limits_{z'\in\overline{Z}} g(x\oplus z).
\]
Write $P(x) = \prod\limits_{z\in Z}{1_{x\oplus z\in \mathcal{U}_{2n}}}\prod\limits_{z\in \overline{Z}}{1_{x\oplus z\in \mathcal{D}_{2n}}}$. Then
\begin{align*}
W_{\leq 2^{100\cdot k}}[F_1\cdots F_{2^{k'-2}}\chi_B]
&=\frac{1}{\mu(\mathcal{U}_{2n})^{2\card{Z}}
\mu(\mathcal{D}_{2n})^{2\card{\overline{Z}}}}
W_{\leq 2^{100\cdot k}}[P\chi_B]\\
&\leq
\frac{1}{\mu(\mathcal{U}_{2n})^{2\card{Z}}
\mu(\mathcal{D}_{2n})^{2\card{\overline{Z}}}}
\E[P]^2 \log^{O_{k}(1)}(\E[P]).
\end{align*}
where the last transition is by Lemma~\ref{lem:lvl_d_inequality}.
By Lemma~\ref{lem:slightly_notrivial_bound} we have
$\E[P]\leq n^{-\card{Z}/2} (\log n)^{O_k(1)}$, and
by inspection $\mu(\mathcal{D}_{2n})\geq \Omega_k(1)$,
$\mu(\mathcal{U}_{2n})\geq \Omega(n^{-1/2})$. Thus,
using the monotonicity of the function $z\rightarrow z^2\log^{O_{k}(1)}(1/z)$ in the interval $[0,c]$ where
$c = c(k)>0$ is an absolute constant, we get that
\[
W_{\leq 2^{100\cdot k}}[F_1\cdots F_{2^{k-2}}\chi_B]
\leq (\log n)^{O_k(1)}.
\]
It follows that
\begin{equation}\label{eq6}
    W_B\leq (\log n)^{O_k(1)},
\end{equation}
and plugging~\eqref{eq6} into~\eqref{eq4} gives that
\[
\Expect{x_1,\ldots,x_{k'-2}}{1_E
\cdot \sum\limits_{S\in T_2}\card{\widehat{h}(S)}^4}
\leq \frac{(\log n)^{O_k(1)}}{n}.
\]
as required.
\end{proof}

Next, we bound the contribution from $T_3$.
\begin{claim}\label{claim:T_3}
   $\Expect{x_1,\ldots,x_{k'-2}}{1_E
\cdot \sum\limits_{S\in T_3}\card{\widehat{h}(S)}^4} \leq O_k\left(\frac{1}{n}\right)$.
\end{claim}
\begin{proof}
Fix $x_1,\ldots,x_{k'-2}$ satisfying $E$ and fix $S'\in T_3$. Note that
\[
\card{{\sf orb}_{S'}}\card{\widehat{h}(S')}^2
=
\sum\limits_{S\in {\sf orb}_{S'}}\card{\widehat{h}(S)}^2
\leq \norm{h}_2^2,
\]
where we used Parseval's equality. Note that $\norm{f-g}_{\infty}\leq n$, and so
$\norm{h}_2^2\leq \norm{f-g}_{\infty}^{2^{k-1}}
\leq n^{2^{k-1}}$. Concluding, we get via Claim~\ref{claim:orbit_estimates} that
\[
\card{\widehat{h}(S')}^2
\leq \frac{n^{2^{k-1}}}{\card{{\sf orb}_{S'}}}
\leq O_{k}(n^{-99\cdot 2^k}).
\]
It follows that
\[
\Expect{x_1,\ldots,x_{k'-2}}{1_E
\cdot \sum\limits_{S\in T_3}\card{\widehat{h}(S)}^4}
\leq
O_{k}(n^{-99\cdot 2^k})
\Expect{x_1,\ldots,x_{k'-2}}{\sum\limits_{S\in T_3}\card{\widehat{h}(S)}^2}
\leq
O_{k}(n^{-99\cdot 2^k})
\Expect{x_1,\ldots,x_{k'-2}}{\norm{h}_2^2},
\]
which is at most
$O_{k}(n^{-98\cdot 2^k})\leq O_k\left(\frac{1}{n}\right)$ using the upper bound
$\norm{h}_2^2\leq n^{2^{k-1}}$.
\end{proof}

We end by upper bounding the contribution from $T_1$:
\begin{claim}\label{claim:T_1}
   $\Expect{x_1,\ldots,x_{k'-2}}{1_E
\cdot \sum\limits_{S\in T_1}\card{\widehat{h}(S)}^4}
\leq \frac{(\log n)^{C_{k'-1,k} + O_k(1)}}{n}.$
\end{claim}
\begin{proof}
Fix $x_1,\ldots,x_{k'-2}$ satisfying $E$.
We observe that for all $S\in T_1$ it holds that
$\card{\widehat{h}(S)} = \card{\widehat{h}(\emptyset)}$.
To see that, it suffices to show that $\chi_S(x)$ is constant
on all $x$ such that $h(x)\neq 0$. It suffices
to prove this assertion on any atom $S$ of $\mathcal{B}$,
and for the sake of simplicity of notation we take
$S = {\sf supp}(x_1)\cap\ldots\cap {\sf supp}(x_{k'-2})$.

By definition of $h$, if $h(x)\neq 0$, then we must have
that $x\oplus z\in \mathcal{D}_{2n}$ for all $z\in {\sf Span}(x_1,\ldots,x_{k'-2})$. Thus, we get that for any such $x$
it holds that $\card{x\oplus z} = a\pmod{2^k}$, i.e. the Hamming
weight of $x\oplus z$ is constant modulo $2^k$.
Applying
Fact~\ref{fact:trivial_inter} on $x\oplus z$, $x$ and $z$,
we conclude that the Hamming weight of $x\land z$ modulo $2^{k-1}$ depends
only on $\card{z}$. In particular, for
$z\in {\sf Span}(x_2,\ldots,x_{k'-2})$
the Hamming weight of
$x\land x_1$ and $x\land (x_1\oplus z)$ is constant
modulo $2^{k-1}$, and as
\[
\card{x\land (x_1\oplus z)}
=
\card{x\land (x_1 \cup z)}
-
\card{x\land x_1\land z}
=\card{x\land z}
+\card{x\land z}
-2\card{x\land x_1\land z}
\]
we get that the Hamming weight of $\card{x\land x_1\land z}$
modulo $2^{k-2}$ depends only on $\card{z}$. Continuing in this fashion, we conclude that the Hamming weight of
$\card{x\land x_1\land\ldots\land x_{k-2}}$ modulo $2$ is
constant, and hence
$\chi_S(x)$ is constant.

Thus, we get that
\[
\Expect{x_1,\ldots,x_{k'-2}}{1_E
\cdot \sum\limits_{S\in T_1}\card{\widehat{h}(S)}^4}
\leq
2^{2^{k-2}}
\Expect{x_1,\ldots,x_{k'-2}}{\card{\widehat{h}(\emptyset)}^4}.
\]
Observe now that using the notations of Claim~\ref{claim:T2}
we have that
\[
\card{\widehat{h}(\emptyset)}^2
\leq W_{\emptyset}
\leq (\log n)^{O_k(1)}
\]
where we used~\eqref{eq6}. Thus,
\[
\Expect{x_1,\ldots,x_{k'-2}}{1_E
\cdot \sum\limits_{S\in T_1}\card{\widehat{h}(S)}^4}
\leq
(\log n)^{O_k(1)}
\Expect{x_1,\ldots,x_{k'-2}}{\card{\widehat{h}(\emptyset)}^2}
=(\log n)^{O_k(1)}\norm{f-g}_{U_{k'-1}}^{2^{k'-1}}.
\]
Upper bounding $\norm{f-g}_{U_{k'-1}}^{2^{k'-1}}$
by $\frac{(\log n)^{C_{k'-1,k}}}{n}$ via the induction hypothesis
finishes the proof.
\end{proof}
Plugging Claims~\ref{claim:T_1},~\ref{claim:T2},~\ref{claim:T_3}
into~\eqref{eq2} finishes the proof of
the inductive step, and thereby the proof of
Theorem~\ref{thm:main'}.

\section{Applications}
In this section we use Theorem~\ref{thm:main}
to prove a few applications
in property testing. Towards
this end, we use~\cite[Corollary 4.4]{dodos2022uniformity}, which
we specialize below for the
group $(\{0,1\}^{n},\oplus)$.
\begin{thm}\label{thm:dense_model_gowers}
    For all $k\in\mathbb{N}$ and
    $\eps>0$ there exists $\delta>0$
    such that the following holds.
    Suppose that $\nu\colon \{0,1\}^n\to [0,\infty)$ is a function such that
    $\norm{\nu-1}_{U_{2k}}\leq \delta$.
    Then for all $f\colon \{0,1\}^n\to\mathbb{R}$ such
    that $\card{f(x)}\leq \nu(x)$
    pointwise, there exists $\tilde{f}\colon \{0,1\}^n\to[-1,1]$
    such that
    \[
        \norm{f-\tilde{f}}_{U_k}\leq \eps.
    \]
\end{thm}
\subsection{Proof of Theorem~\ref{thm:linearity}}
In this section we prove Theorem~\ref{thm:linearity},
and for that we first need the
following claim which is a version
of that theorem for bounded functions
on $\{0,1\}^n$.
\begin{claim}\label{claim:trivial_blr}
    Suppose that a function
    $f\colon \{0,1\}^n\to[-1,1]$
    satisfies that
    \[
        \Expect{x,y,z\in\{0,1\}^n}{f(x)f(y)f(z)f(x\oplus y\oplus z)}\geq \eps.
    \]
    Then there exists $S\subseteq [n]$
    such that $\card{\widehat{f}(S)}\geq \sqrt{\eps}$.
\end{claim}
\begin{proof}
Plugging in the Fourier expansion of $f$,
we get that
\begin{align*}
\eps
\leq
\Expect{x,y,z\in\{0,1\}^n}{f(x)f(y)f(z)f(x\oplus y\oplus z)}
=\sum\limits_{S\subseteq [n]}\widehat{f}(S)^4
&\leq \max_{S}\widehat{f}(S)^2
\sum\limits_{S\subseteq [n]}\widehat{f}(S)^2\\
&=\max_{S}\widehat{f}(S)^2\norm{f}_2^2,
\end{align*}
which is at most $\max_{S}\widehat{f}(S)^2$
as $\norm{f}_2\leq 1$.
\end{proof}

We also need the Cauchy-Schwarz-Gowers inequality (see for example~\cite[Lemma 4.2]{hatami2011higher})
\begin{lemma}\label{lem:CSG}
  Suppose that $\{f_S\}_{S\subseteq [r]}$
  are complex valued functions. Then
  \[
  \card{\Expect{x, y_1,\ldots,y_r}{\prod\limits_{S\subseteq [r]}C^{k-\card{S}}f_S\left(x+\sum\limits_{i\in S}y_i\right)}}
  \leq \prod\limits_{S\subseteq [r]}\norm{f_S}_{U^r}.
  \]
\end{lemma}

We are now ready to prove Theorem~\ref{thm:linearity}.
\begin{proof}[Proof of Theorem~\ref{thm:linearity}]
    Fix $k=2$, $\eps>0$ and
    take $\delta$ from Theorem~\ref{thm:dense_model_gowers}
    for $\eps/10^6$. Take a function
    $f\colon \mathcal{U}_{2n}\to\{0,1\}$
    and define $f'\colon \{0,1\}^{2n}\to\mathbb{R}$
    by $f'(x) = (-1)^{f(x)}
    1_{x\in\mathcal{U}_{2n}}
    \frac{\E[1_{\mathcal{D}_{2n,2k}}]}{\E[1_{\mathcal{U}_{2n}}]}$.
    Note that
    \begin{align*}
    &\Expect{x,y,z}{f'(x)f'(y)f'(z)f'(x\oplus y\oplus z)}
    =
    \E[1_{\mathcal{D}_{2n,2k}}]^4
    \frac{\Prob{x,y,z\in \{0,1\}^{2n}}{x,y,z,x\oplus y\oplus z\in\mathcal{U}_{2n}}}{\E[1_{U_{2n}}]^4}\\
    &\qquad\qquad\qquad\qquad\cdot \left(2\cProb{x,y,z}{x,y,z,x\oplus y\oplus z\in \mathcal{U}_{2n}}{f(x)+f(y)+f(z) = f(x\oplus y\oplus z)}-1\right).
    \end{align*}
    By the premise the second term is
    at least $2\eps$, and a
    direct calculation shows that
    \[
    \Prob{x,y\in \{0,1\}^{2n}}{x,y,z,x\oplus y\oplus z\in\mathcal{U}_{2n}}\geq \E[1_{U_{2n}}]^4,
    \]
    and hence $\Expect{x,y,z}{f'(x)f'(y)f'(z)f'(x\oplus y\oplus z)}\geq 2\E[1_{\mathcal{D}_{2n,2k}}]\eps$.
    As $\E[1_{\mathcal{D}_{2n,2k}}]\geq \frac{1}{10}$ we get that

   \begin{equation}\label{eq7}
    \Expect{x,y,z}{f'(x)f'(y)f'(z)f'(x\oplus y\oplus z)}\geq \frac{2\eps}{10^4}.
    \end{equation}
    Take
    \[
    \nu = \frac{\E[1_{\mathcal{D}_{2n,2k}}]}{\E[1_{\mathcal{U}_{2n}}]}1_{\mathcal{U}_{2n}} + 1 - 1_{\mathcal{D}_{2n,2k}},
    \]
    and note that $\nu$ is non-negative and that $\card{f'(x)}\leq \nu(x)$
    pointwise. By Theorem~\ref{thm:main}
    we have
    \[
    \norm{\nu - 1}_{U_{2k}}
    =
    \left\|
    \frac{\E[1_{\mathcal{D}_{2n,2k}}]}{\E[1_{\mathcal{U}_{2n}}]}
    1_{\mathcal{U}_{2n}} - 1_{\mathcal{D}_{2n,2k}}
    \right\|_{U_{2k}}
    \leq
        \left\|
    \frac{1}{\E[1_{\mathcal{U}_{2n}}]}
    1_{\mathcal{U}_{2n}} - \frac{1_{\mathcal{D}_{2n,2k}}}{\E[1_{\mathcal{D}_{2n,2k}}]}
    \right\|_{U_{2k}}
    =o(1),
    \]
    and hence for sufficiently large
    $n_0$, if $n\geq n_0$ we have that
    $\norm{\nu - 1}_{U_{2k}}\leq \delta$.
    Applying Theorem~\ref{thm:dense_model_gowers}
    we get that there is a function
    $\tilde{f}\colon \{0,1\}^n\to[-1,1]$
    with $\norm{f'-\tilde{f}}_{U_k}\leq \frac{\eps}{10^6}$. Define the
    function $\Delta(x) = f'(x) - \tilde{f}(x)$; by the triangle inequality
    we may bound
    \begin{align}\label{eq8}
    &\card{\Expect{x,y,z}{f'(x)f'(y)f'(z)f'(x\oplus y\oplus z)}
    -\Expect{x,y,z}{
    \tilde{f}(x)\tilde{f}(y)
    \tilde{f}(z)\tilde{f}(x\oplus y\oplus z)}}\notag\\
    &=\card{\Expect{x,y,z}{(\tilde{f}+\Delta)(x)(\tilde{f}+\Delta)(y)(\tilde{f}+\Delta)(z)(\tilde{f}+\Delta)(x\oplus y\oplus z)}
    -\Expect{x,y,z}{
    \tilde{f}(x)\tilde{f}(y)
    \tilde{f}(z)\tilde{f}(x\oplus y\oplus z)}}\notag\\
    &\leq
    \Expect{x,y,z}{
    4\Delta(x)\tilde{f}(y)\tilde{f}(z)\tilde{f}(x\oplus y\oplus z)
    +6\Delta(x)\Delta(y) \tilde{f}(z)\tilde{f}(x\oplus y\oplus z)
    +4\Delta(x)\Delta(y)\Delta(z)\tilde{f}(x\oplus y\oplus z)
    }\notag\\
    &+\Expect{x,y,z}{\Delta(x)\Delta(y)\Delta(z)\Delta(x\oplus y\oplus z)
    }.
    \end{align}
    We use Lemma~\ref{lem:CSG} to bound each expectation
    by $\norm{\Delta}_{U_2}$,
    and we
    take the first expression for example. Note that
    the distributions of $(x,y,z,x\oplus y\oplus z)$
    is the same as of $(x,x\oplus y, x\oplus z, x\oplus y\oplus z)$ and that $\Delta,\tilde{f}$ are real valued, so by Lemma~\ref{lem:CSG} the expectation of the first term is at most
    \[
    \norm{\Delta}_{U_2}
    \norm{\tilde{f}}_{U_2}^{3}
    \leq
    \norm{\Delta}_{U_2},
    \]
    where we used the fact that
    $\tilde{f}$ is $1$-bounded.
    Therefore, combining~\eqref{eq7}~\eqref{eq8} we get that
    \[
        \Expect{x,y,z}
        {
        \tilde{f}(x)
        \tilde{f}(y)
        \tilde{f}(z)
        \tilde{f}(x\oplus y\oplus z)}\geq
        \frac{2\eps}{10^4}
        -15\norm{\Delta}_{U_2}
        \geq \frac{\eps}{10^4}.
    \]
    By Claim~\ref{claim:trivial_blr}
    we conclude that there exists
    $S\subseteq [2n]$ such that
    $\card{\widehat{\tilde{f}}(S)}\geq \frac{\sqrt{\eps}}{100}$, and so
    \[
    \card{\widehat{f'}(S)}
    \geq
    \card{\widehat{\tilde{f}}(S)}
    -\norm{f'-\tilde{f}}_{1}
    \geq
    \card{\widehat{\tilde{f}}(S)}
    -\norm{\Delta}_{U_2}
    \geq \frac{\sqrt{\eps}}{200}.
    \]
    Define $L_S = \bigoplus_{i\in S} x_i$,
    and suppose without loss of generality
    that $\widehat{f'}(S)$ is
    positive (otherwise we multiply $\tilde{f}$ by $-1$). Note that
    \[
        \frac{\sqrt{\eps}}{200}
        \leq\widehat{f'}(S)
        =\E[1_{\mathcal{D}_{2n, 2k}}]
        \Expect{x\in\mathcal{U}_{2n}}
        {(-1)^{L_S(x) + f(x)}}
        =
        \E[1_{\mathcal{D}_{2n, 2k}}]
        \left(2\Prob{x\in\mathcal{U}_{2n}}{f(x) = L_S(x)} - 1\right).
    \]
    Rearranging and using
    $\E[1_{\mathcal{D}_{2n, 2k}}]\leq 1$
    finishes the proof.
\end{proof}

\subsection{Proof of Theorem~\ref{thm:gowers_low_soundness}}
We first need the following elementary
conclusion of Theorem~\ref{thm:main}.
\begin{fact}\label{fact:use_gowers_CS}
    \[
    \Prob{x,h_1,\ldots,h_{d}\in\{0,1\}^{2n}}
    {x\oplus \bigoplus_{i\in T}h_i\in\mathcal{U}_{2n}~\forall T\subseteq[d]}
    \geq
    \E[1_{\mathcal{U}_{2n}}]^{2^{d}}.
    \]
\end{fact}
\begin{proof}
Define $f(x) = \frac{1_{x\in\mathcal{U}_{2n}}}{\E[1_{x\in\mathcal{U}_{2n}}]}$, and
note that the left hand side is equal to
\[
\E[1_{\mathcal{U}_{2n}}]^{2^d}
\Expect{x,h_1,\ldots,h_{d}}
{\prod\limits_{T\subseteq [d]}f\left(x\oplus \bigoplus_{i\in T}h_i\right)}
=
\E[1_{\mathcal{U}_{2n}}]^{2^d}
\norm{f}_{U_{d}}^{2^d}
\geq
\E[1_{\mathcal{U}_{2n}}]^{2^d}
\norm{f}_{U_{1}}^{2^d}
=
\E[1_{\mathcal{U}_{2n}}]^{2^d},
\]
where we used the fact
that $\norm{f}_{U_d}\geq \norm{f}_{U_1}$
and that $\norm{f}_{U_1} = 1$.
\end{proof}

Towards the proof of Theorem~\ref{thm:gowers_low_soundness}
we shall need the inverse Gowers
theorem over fields with small
characteristics due to~\cite{tao2012inverse}.
\begin{thm}\label{thm:inverse_gowers_low_soundness}
For all $\eps>0$ there is $\delta>0$
such that if $f\colon \{0,1\}^n\to[-1,1]$
satisfies that $\norm{f}_{U_d}\geq \eps$. Then there exists a non classical polynomial $p\colon \{0,1\}^n\to \mathbb{T}$ of degree at
most $d-1$ such that
\[
\card{\inner{f}{e^{2\pi{\bf i} p}}}\geq \delta.
\]
\end{thm}
We are now ready to prove
Theorem~\ref{thm:gowers_low_soundness}.
\begin{proof}[Proof of Theorem~\ref{thm:gowers_low_soundness}]
We take parameters
\[
0< \xi\ll \eta\ll \delta \ll \eps\ll d^{-1} = k^{-1}\leq 1.
\]
    Take a function
    $f\colon \mathcal{U}_{2n}\to\{0,1\}$
    and define $f'\colon \{0,1\}^{2n}\to\mathbb{R}$
    by
    \[
    f'(x) = (-1)^{f(x)}
    1_{x\in\mathcal{U}_{2n}}
    \frac{\E[1_{\mathcal{D}_{2n,2d}}]}{\E[1_{\mathcal{U}_{2n}}]}.
    \]
    Note that
    \begin{align*}
    \norm{f'}_{U_{d}}^{2^{d}}
    &=
    \E[1_{\mathcal{D}_{2n,2d}}]^{2^{d}}
    \frac{\Prob{x,h_1,\ldots,h_{d}\in \{0,1\}^{2n}}{x\oplus\bigoplus_{i\in T}h_i\in\mathcal{U}_{2n}~\forall T\subseteq [d]}}{\E[1_{U_{2n}}]^{2^{d}}}\\
    &\cdot
    \left(2
    \cProb{x,h_1,\ldots,h_{d}\in \{0,1\}^{2n}}{x\oplus\bigoplus_{i\in T} h_i\in\mathcal{U}_{2n}~\forall T\subseteq [d]}{\sum_{T}f(x\oplus\bigoplus_{i\in T}h_i)=0} - 1\right).
    \end{align*}
    The first term is at least
    $\Omega_{d}(1)$,
    the second term is at least
    $1$ by Fact~\ref{fact:use_gowers_CS} and
    the third term is at least
    $2\eps$ by the premise.
    Thus, $\norm{f'}_{U_{d}}^{2^{d}}\geq
    \Omega_{d}(\eps)$.
    Take
    \[
    \nu = \frac{\E[1_{\mathcal{D}_{2n,2d}}]}{\E[1_{\mathcal{U}_{2n}}]}1_{\mathcal{U}_{2n}} + 1 - 1_{\mathcal{D}_{2n,2k}},
    \]
    and note that $\nu$ is non-negative and that $\card{f'(x)}\leq \nu(x)$
    pointwise. By Theorem~\ref{thm:main}
    we have
    \[
    \norm{\nu - 1}_{U_{2d}}
    =
    \left\|
    \frac{\E[1_{\mathcal{D}_{2n,2d}}]}{\E[1_{\mathcal{U}_{2n}}]}
    1_{\mathcal{U}_{2n}} - 1_{\mathcal{D}_{2n,2d}}
    \right\|_{U_{2d}}
    \leq
        \left\|
    \frac{1}{\E[1_{\mathcal{U}_{2n}}]}
    1_{\mathcal{U}_{2n}} - \frac{1_{\mathcal{D}_{2n,2d}}}{\E[1_{\mathcal{D}_{2n,2d}}]}
    \right\|_{U_{2d}}
    =o(1),
    \]
    and hence for sufficiently large
    $n_0$, if $n\geq n_0$ we have that
    $\norm{\nu - 1}_{U_{2d}}\leq \xi$.
    Applying Theorem~\ref{thm:dense_model_gowers}
    we get that there is a function
    $\tilde{f}\colon \{0,1\}^n\to[-1,1]$
    with $\norm{f'-\tilde{f}}_{U_d}\leq \eta$. It follows from the triangle
    inequality that
    \[
    \norm{\tilde{f}}_{U_{d}}
    \geq
    \norm{f'}_{U_{d}}
    -\norm{f'-\tilde{f}}_{U_{d}}
    \geq
    \Omega_d(\eps^{2^{-d}})
    -
    \eta
    \geq
    \Omega_d(\eps^{2^{-d}}).
    \]
    By Theorem~\ref{thm:inverse_gowers_low_soundness} we conclude that there
    exists a non-classical polynomial
    $p\colon \{0,1\}^{n}\to \mathbb{T}$
    of degree at most $d-1$ such
    that
    \[
    \card{\inner{\tilde{f}}{e^{2\pi{\bf i}p}}}\geq\delta.
    \]
    Thus, by the triangle inequality
    \[
    \card{\inner{f'}{e^{2\pi{\bf i}p}}}
    \geq
    \card{\inner{\tilde{f}}{e^{2\pi{\bf i}p}}}
    -
    \card{\inner{f'-\tilde{f}}{e^{2\pi{\bf i}p}}}
    \geq
    \delta
    -
    \card{\inner{f'-\tilde{f}}{e^{2\pi{\bf i}p}}}.
    \]
    Note that
    \[
        \card{\inner{f'-\tilde{f}}{e^{2\pi{\bf i}p}}}
        =
        \norm{(f'-\tilde{f})e^{-2\pi{\bf i}p}}_{U_1}
        \leq
        \norm{(f'-\tilde{f})e^{-2\pi{\bf i}p}}_{U_d}
        =\norm{f'-\tilde{f}}_{U_d},
    \]
    where in the last transition we used
    the fact that $p$ is a non-classical
    polynomial of degree at most $d-1$
    and hence its order $d$ derivatives vanish. As
    $\norm{f'-\tilde{f}}_{U_d} \leq \eta$ we conclude
    that $\card{\inner{f'}{e^{2\pi{\bf i}p}}}\geq \frac{\delta}{2}$,
    and so
    \[
    \card{\Expect{x\in \mathcal{U}_{2n}}
    {
    (-1)^{f(x)}
    e^{-2\pi {\bf i} p(x)}}}\geq \frac{\delta}{2}.
    \qedhere
    \]
\end{proof}

\subsection{Proof of Theorem~\ref{thm:biased_rank}}
The proof requires~\cite[Theorem 1.20]{tao2012inverse},
stated below for convenience.
\begin{thm}\label{thm:GIP}
Let $\delta>0$,
and suppose that
$p\colon \{0,1\}^{2n}\to \mathbb{T}$ is a non-classical
polynomial of degree
$d$ such that $\|e^{2\pi{\bf i} p}\|_{U_{d}}\geq \delta$. Then there exists $L = L(\delta, d)$, non-classical polynomials $Q_1,\ldots,Q_{L}$ of degree at most $d-1$
and $F\colon [0,1)^{L}\to \mathbb{T}$
such that for all $x\in\{0,1\}^{2n}$
we have $P(x) = F(Q_1(x),\ldots,Q_L(x))$.
\end{thm}
We are now ready to prove Theorem~\ref{thm:biased_rank}, and we assume for convenience of notations that $n=0\pmod{2^{d}}$.
Note that
\begin{align*}
\card{\Expect{x}{e^{2\pi{\bf i}P(x)}
\left(
\frac{1_{\mathcal{U}_{2n}}(x)}{\E[1_{\mathcal{U}_{2n}}]}
-
\frac{1_{\mathcal{D}_{2n,d+1}}(x)}{\E[1_{\mathcal{D}_{2n,d+1}}]}
\right)}}
&=\left\|{e^{2\pi{\bf i}P}
\left(
\frac{1_{\mathcal{U}_{2n}}}{\E[1_{\mathcal{U}_{2n}}]}
-
\frac{1_{\mathcal{D}_{2n,d+1}}}{\E[1_{\mathcal{D}_{2n,d+1}}]}
\right)}\right\|_{U_{1}}\\
&\leq
\left\|{e^{2\pi{\bf i}P}
\left(
\frac{1_{\mathcal{U}_{2n}}}{\E[1_{\mathcal{U}_{2n}}]}
-
\frac{1_{\mathcal{D}_{2n,d+1}}}{\E[1_{\mathcal{D}_{2n,d+1}}]}
\right)}\right\|_{U_{d+1}}\\
&=
\left\|
\frac{1_{\mathcal{U}_{2n}}}{\E[1_{\mathcal{U}_{2n}}]}
-
\frac{1_{\mathcal{D}_{2n,d+1}}}{\E[1_{\mathcal{D}_{2n,d+1}}]}\right\|_{U_{d+1}}\\
&=o(1),
\end{align*}
where we used Theorem~\ref{thm:main}.
Thus, as
the assumption of the
theorem implies that
$\card{\Expect{x}{\frac{e^{2\pi{\bf i} P(x)}}{\E[1_{\mathcal{U}_{2n}}]}
1_{\mathcal{U}_{2n}}(x)}}\geq \delta$, we conclude that
\[
\card{\Expect{x}{\frac{e^{2\pi{\bf i} P(x)}}{\E[1_{\mathcal{D}_{2n, d+1}}]}
1_{\mathcal{D}_{2n, d+1}(x)}}}\geq \delta - o(1)
\geq \frac{\delta}{2}.
\]
By definition,
\[
1_{\mathcal{D}_{2n, d+1}}(x)
=1_{\card{x} = 0\pmod{2^d}}
=
\frac{1}{2^d}
\sum\limits_{j=0}^{2^{d}-1}
e^{2\pi{\bf i} \frac{j\card{x}}{2^d}},
\]
and plugging this above implies that there exists $j$ such that $\card{\Expect{x}{e^{2\pi{\bf i}(P(x) + j\card{x}/2^d)}}}
\geq \frac{\delta}{2}$.
Note that the polynomial
$P(x) + j\card{x}/2^d$
is a non-classical
polynomial of degree $d$.
The conclusion is now immediate from Theorem~\ref{thm:GIP}.
\section{Acknowledgements}
We thank Zach Hunter for carefully reading the
paper, giving us useful feedback and pointing
a simplification in the proof of Fact~\ref{fact:use_gowers_CS}.
We thank an anonymous referee for helpful comments.
\bibliography{ref}
\bibliographystyle{plain}
\end{document}